\documentclass[11pt]{amsart}
\usepackage{amsfonts,amsmath,amssymb,amsthm,cmtiup} 
\usepackage{mathrsfs}
\usepackage[noadjust]{cite}

\newtheorem{theorem}{Theorem}
\newtheorem*{Kunenslemma}{Kunen's lemma}
 
\newtheorem{corollary}{Corollary} 
\newtheorem{proposition}{Proposition} 

\theoremstyle{definition} 
\newtheorem{remark}{Remark}
\newtheorem{definition}{Definition}

\let\le\leqslant
\let\ge\geqslant

\newcommand{\cR}{\mathscr{R}}
\newcommand{\pl}{\operatorname{\mbox{$p$-$\mathrm{lim}$}}}
\newcommand{\ql}{\operatorname{\mbox{$q$-$\mathrm{lim}$}}}
\newcommand{\leRK}{\mathrel{\le_{\mathrm{RK}}}}

\newcommand{\leRF}{\mathrel{\le_{\mathrm{RF}}}}
\newcommand{\leRB}{\mathrel{\le_{\mathrm{RB}}}}

\begin{document} 

\title[Discrete Ultrafilters and Homogeneity of Product Spaces]{Discrete Ultrafilters\\ 
and Homogeneity of Product Spaces} 
\author{Anastasiya Groznova}

\address{Department of General Topology and Geometry, Faculty of Mechanics and  Mathematics, 
M.~V.~Lomonosov Moscow State University, Leninskie Gory 1, Moscow, 199991 Russia}

\email{avsav999@mail.ru}

\author{Ol'ga Sipacheva}

\address{Department of General Topology and Geometry, Faculty of Mechanics and  Mathematics, 
M.~V.~Lomonosov Moscow State University, Leninskie Gory 1, Moscow, 199991 Russia}

\email{osipa@gmail.com}

\begin{abstract}
An ultrafilter $p$ on $\omega$ 
is said to be discrete if, given any function $f\colon \omega \to X$ to any completely regular Hausdorff space, 
there is an $A \in p$ such that $f(A)$ is discrete. Basic properties of discrete ultrafilters are studied. Three 
intermediate classes of spaces $\mathscr R_1 \subset \mathscr R_2 \subset \mathscr R_3$ between the class 
of $F$-spaces and the class of van~Douwen's $\beta\omega$-spaces are introduced. It is proved that no product of 
infinite compact $\mathscr R_2$-spaces is homogeneous; moreover, under the assumption $\mathfrak d =\mathfrak c$, 
no product of $\beta\omega$-spaces is homogeneous.
\end{abstract}

\keywords{Discrete ultrafilter, Rudin--Frolik order, Rudin--Keisler order, Rudin--Blass order, 
$\beta\omega$-space, homogeneous product space}

\subjclass[2020]{54H11, 54G05, 03E35}

\maketitle

In \cite{Frolik} Frol\'\i k proved the nonhomogeneity of the Stone--\v Cech remainder 
$\omega^*=\beta\omega\setminus \omega$ of $\omega$ by noticing that if two discrete sequences in $\omega^*$ 
converge to the same point $x\in \omega^*$ along ultrafilters $p$ and $q$, then $p$ and $q$ are compatible 
in the Rudin--Frol\'\i k order. The idea is quite natural: if $p, q\in 
\omega^*$ are incompatible and $D=\{d_n:n\in \omega\}$ is a countable discrete subset of $\omega^*$, then 
there cannot exist a homeomorphism $h\colon \omega^*\to\omega^*$ taking $\ql_n d_n$ to $\pl_n d_n$, because 
if it existed, then $(d_n)_{n\in \omega}$ and $(h(d_n))_{n\in\omega}$ would be discrete sequences converging to 
the same point $\pl_n d_n$ along $p$ and $q$, respectively. 

Frol\'\i k's idea of proving nonhomogeneity by considering orderings of ultrafilters was developed by Kunen. 
One of the key ingredients in his proof of the inhomogeneity of any product of infinite 
compact $F$-spaces is the following lemma on the Rudin--Keisler comparability of ultrafilters.  

\begin{Kunenslemma}[{\cite[Lemma 4]{K-OPIT}}]
Let $p, q \in \omega^*$ be Rudin--Keisler incomparable weak $P$-points, and let $X$ be
any compact $F$-space. Suppose that $x\in X$,  $(d_m)_{m\in \omega}$ is a discrete sequence of distinct points in 
$X$, $(e_n)_{n\in \omega}$ is any sequence of points in $X$, and $x = \pl_m d_m = \ql_n e_n$. 
Then $\{n : e_n = x\} \in q$. 
\end{Kunenslemma}

The present paper arose from an attempt to extend Kunen's lemma and, thereby, his result on the 
inhomogeneity of product spaces to other classes of spaces. This can be done  
by looking for larger classes for which Kunen's argument still works 
or by strengthening the assumptions on the ultrafilters $p$ and~$q$. In 
\cite{GS} we introduced new classes $\cR_1$, $\cR_2$, and $\cR_3$ of topological spaces, which lie strictly 
between the classes of $F$- and $\beta \omega$-spaces, and proved that Kunen's lemma remains valid for 
$\cR_2$-spaces. In this paper we mainly focus on special ultrafilters, namely, discrete ones, which are especially 
interesting in the context of homogeneity, because the convergence of any sequence along a discrete ultrafilter 
reduces to the convergence of a discrete subsequence. 

\section{Preliminaries}

Throughout the paper by a space we mean a completely regular Hausdorff topological space. 

Given a space $X$ and $A\subset X$, by $\overline A$ we denote the closure of $A$ in $X$ and by $|A|$, the 
cardinality of $A$. Recall that sets $A,B\subset X$ are \emph{separated} if $\overline{A}\cap B = A\cap 
\overline{B} =\varnothing$. A subspace $A\subset X$ is \emph{$C^*$-embedded} in $X$ if any continuous function 
$f\colon A\to [0,1]$ has a continuous extension $\hat f\colon X\to [0,1]$. We say that a sequence 
$(x_n)_{n\in\omega}$ of points of a space $X$ is \emph{discrete} if its range $\{x_n:n\in \omega\}$ is a discrete 
(not necessarily closed) subspace of~$X$.

Every space $X$ has \emph{Stone--\v Cech compactification} $\beta X$; this is a compact space in which $X$ is 
densely embedded so that any continuous map $f\colon X\to K$ to a Hausdorff compact space $K$ has a continuous 
extension $\beta f\colon \beta X\to K$.

We use the standard notation $\mathbb R$ for the real line with the usual topology, $\mathbb Q$ for the space of 
rationals, $\omega$ for the set of nonnegative integers (endowed with the discrete topology when appropriate), 
$\beta \omega$ for the Stone--\v Cech compactification of $\omega$, and $\omega^*$ for the Stone--\v Cech 
remainder $\beta\omega\setminus \omega$. It is well known that $\beta\omega$ is nothing but the space of 
ultrafilters on $\omega$ endowed with the topology generated by the base consisting of sets of the form $\overline 
A=\{p\in \beta\omega: A\in p\}$, where $A\subset \omega$; each $n\in \omega$ is identified with the principal 
ultrafilter $p(n)=\{A\subset \omega: n\in A\}$, so that $\omega^*$ is precisely the subspace of nonprincipal 
ultrafilters and $\omega$ is embedded in $\beta \omega$ as a dense open discrete subspace. Moreover, for each 
$A\subset \omega$, the set $\overline A$ defined above is indeed the closure of $A$ in $\beta\omega$. Note also 
that an ultrafilter is nonprincipal if and only if all of its elements are infinite.

The \emph{limit} of a sequence $(x_n)_{n\in \omega}$ in a space $X$ \emph{along an ultrafilter} $p$ on $\omega$, 
or the \emph{$p$-limit} of $(x_n)_{n\in \omega}$, is a point $x\in X$, denoted by $\pl_n x_n$, such that, for 
any neighborhood $U$ of $x$, the set $\{n\in \omega: x_n\in U\}$ belongs to $p$. We say that $x$ is the 
\emph{nontrivial} $p$-limit of $(x_n)_{n\in \omega}$ if $\{n\in \omega: x_n=x\}\notin p$. A sequence may have no 
$p$-limit, but if its $p$-limit exists, then it is unique (recall that we assume all spaces under consideration to 
be Hausdorff). Moreover, if $X$ is compact, then any sequence $(x_n)_{n\in \omega}$ in $X$ has precisely one 
$p$-limit for any $p\in \beta\omega$. To see this, it suffices to consider the continuous extension $\beta f\colon 
\beta \omega\to X$ of the map $f\colon \omega\to X$ defined by $f(n)=x_n$ for $n\in \omega$ and apply the 
following remark.

\begin{remark}
\label{Remark1}
(i)\enspace 
Any ultrafilter $p\in \beta\omega$ is the $p$-limit of the sequence $(n)_{n\in\omega}$ in $\beta\omega$: $p=\pl_n 
n$. 

(ii)\enspace 
Let $f\colon X\to Y$ be a continuous map of spaces $X$ and $Y$, and let $x, x_n\in X$ for $n\in 
\omega$. If $p\in \beta\omega$ and $x= \pl_n x_n$, then $f(x) = \pl_n f(x_n)$. 

(iii)\enspace
Suppose that $p\in \omega^*$, $\varphi\colon \omega\to\omega$ is any function, $(x_n)_{n\in \omega}$ and 
$(y_k)_{k\in \omega}$ are two sequences in a space $X$, and there is a $P_0\in p$ such that $y_{\varphi(n)}=x_n$ 
for $n\in P_0$. If $x=\pl_n x_n$, then $x = \beta\varphi$-$\lim_k y_k$. Indeed, for any neighborhood $U$ of $x$, 
we have $\{n\in \omega: x_n\in U\}\in p$. Therefore, $\{n\in P_0: y_{\varphi(n)}\in U\}=P\in p$. Finally, $\{k\in 
\omega: y_k\in U\}\supset \varphi(P)\in \beta\varphi(p)$. 
\end{remark}

Any function $f\colon \omega\to \omega$ can be treated as a map $\omega\to \beta\omega$ and, therefore, 
has a continuous extension $\beta f\colon \beta\omega\to \beta\omega$. This extension is explicitly described as  
$$ 
\beta f(p)=\{A\subset \omega: f^{-1}(A)\in p\}, \qquad p\in \beta\omega 
$$ 
(see \cite[Lemma~3.30]{HS}). Ultrafilters 
$p$ and $q$ on $\omega$ are said to be \emph{equivalent} if there exists a bijection $\varphi\colon \omega\to 
\omega$ such that $\beta\varphi(p)=q$. 

\begin{remark}
\label{Remark_add}
Ultrafilters $p$ and $q$ on $\omega$
are equivalent if and only if there exists a function $f\colon \omega\to \omega$ and an $A\in p$ such that 
$\beta f(p)=q$ and the restriction $f\restriction A$ of $f$ to $A$ is one-to-one. Indeed, if such $f$ and $A$ 
exist, then there is a $B\subset A$, $B\in p$, for which $|\omega\setminus B|=|\omega\setminus f(B)|=\omega$, and 
any bijection $g\colon \omega\to \omega$ extending $f\restriction B$ witnesses the equivalence of $p$ and~$q$.
\end{remark}

The extension $\beta f$ of a function $f\colon \omega \to K$ to an arbitrary Hausdorff compact space $K$ has a 
simple description as well: the image under $\beta f$ of an ultrafilter $p$ is defined by $\{\beta f(p)\}= 
\bigcap_{A\in p} \overline{f(A)}$ (see \cite[Theorem~3.27]{HS}).

\section{Orders on $\beta\omega$}

There are several natural order relations on classes of 
equivalent ultrafilters; we will consider Rudin--Keisler, Rudin--Blass, and Rudin--Frol\'\i k orders. For 
detailed information about these and some other orders on $\beta\omega$, see, 
e.g., \cite{CN, vanMill, HS}, and references therein. 

The \emph{Rudin--Keisler order} $\leRK$ on $\beta\omega$ is defined by declaring that, for $p,q\in \beta 
\omega$, $p\leRK q$ if and only if there exists a function $f\colon \omega\to \omega$ such that $\beta f(q)=p$. 

The \emph{Rudin--Blass order} $\leRB$ on $\beta\omega$ is defined by declaring that, for $p,q\in \beta 
\omega$, $p\leRB q$ if and only if there exists a finite-to-one function $f\colon \omega\to \omega$ such that 
$\beta f(q)=p$. 

The \emph{Rudin--Frol\'\i k  order} $\leRF$ on $\beta\omega$ is defined by declaring that, for $p,q\in \beta 
\omega$, $p\leRF q$ if and only if there exists an injective function $\varphi\colon \omega\to \beta\omega$ such 
that $\varphi(\omega)$ is discrete and $\beta \varphi(p)=q$. 

(Note that the map $\varphi$ in the last definition and the maps $f$ in the 
two preceding ones act on the ultrafilters in opposite directions.) 

It is known that all these relations are indeed orders on the equivalence classes of ultrafilters. The relation 
$\leRK$ is largest. Indeed, the implication $p\leRB q$ $\implies$ $p\leRK q$ is obvious, and if 
$p\leRF q$, $\varphi$ is the corresponding function in the definition of $\leRF$, and $A_i$, $i\in \omega$, 
are disjoint subsets of $\omega$ determining disjoint neighborhoods 
$\overline A_i$ of the points $\varphi(i)$ in $\beta \omega$, then the function $f\colon \omega\to \omega$ 
defined by setting $f(n)=i$ if $n\in A_i$ and $f(n)=0$ if $n\notin \bigcup A_i$ witnesses that $p\leRK q$.  

\begin{remark}
\label{remark-added1}
Each nonprincipal ultrafilter on $\omega$ has at most $2^\omega$ \,$\leRK$-predecessors (because the number of 
functions $\omega\to\omega$ is $2^\omega$). 
\end{remark}

\begin{proposition}
\label{Proposition1}
For $p,q\in \beta\omega$, $p\leRF q$ if and only if $q$ is the $p$-limit of a discrete 
sequence $(x_n)_{n\in \omega}$ of distinct points in $\beta \omega$. 
\end{proposition}

\begin{proof}
According to Remark~\ref{Remark1}, we have $q=\pl_n x_n$ if and only if $q=\beta \varphi(p)$ for $\varphi\colon 
\omega\to \beta \omega$ defined by $\varphi(n)=x_n$ for $n\in \omega$. Therefore, if $\{x_n: n\in \omega\}$ is  
discrete and consists of distinct points, then $q=\pl_n x_n$ if and only if the map $\varphi$ witnesses that 
$p\leRF q$. \end{proof}

The following theorem was proved by M.~E.~Rudin in \cite{Rudin} (see also \cite[Theorem~16.16]{CN}).

\begin{theorem}[\cite{Rudin}]
\label{Theorem1}
The set of all $\leRF$-predecessors of any ultrafilter $p\in\beta \omega$ is totally $\leRF$-ordered.
\end{theorem}

Recall that a point $x$ in a topological space $X$ is a \emph{$P$-point} if the intersection of any
countable family of neighborhoods of $x$ is also a neighborhood of $x$, $x$ is a \emph{weak $P$-point} if  
$x\notin \overline D$ whenever $D$ is a countable subset of $X\setminus \{x\}$, and 
$x$ is a \emph{discretely weak $P$-point} if  
$x\notin \overline D$ whenever $D$ is a countable discrete subset of $X\setminus \{x\}$. It is well known that 
$p\in \omega^*$ is a $P$-point in $\omega^*$ if, given any $f\colon \omega \to \omega$, there exists an $A \in p$ 
such that the restriction of $f$ to $A$ is either finite-to-one or constant, or, equivalently, given any 
sequence $(A_n)_{n\in \omega}$ of elements of $p$, there exists an $A\in p$ such that $A\subset ^* A_n$ (i.e., 
$|A\setminus A_n| <\omega$) for each $n\in \omega$. The existence of weak $P$-points in 
$\omega^*$ can be proved in ZFC \cite{Kunen1978}, while the existence of $P$-points in $\omega^*$ is independent 
of ZFC (see, e.g., \cite{Shelah}). Note that there exist discretely weak $P$-points in $\omega^*$ which are not 
weak $P$-points: an example of such a point is any lonely point in the sense of Simon, whose existence in 
$\omega^*$ was proved by Verner in~\cite{Verner}. We also mention that the non-discretely weak $P$-points in 
$\omega^*$ are precisely those of van~Mill's type $A_1$~\cite{vanMill16}.

Related types of ultrafilters are selective and $Q$-point ultrafilters (see, e.g., \cite{Booth}, \cite{Miller}). 
An ultrafilter $p\in \omega^*$ is said to be \emph{selective}, or \emph{Ramsey}, if, given any $f\colon \omega 
\to \omega$, there exists an $A \in p$ such that the restriction of $f$ to $A$ is either one-to-one or constant, 
or, equivalently, given any decreasing sequence $(A_n)_{n\in \omega}$ of elements of $p$, there exists an $A\in p$ 
such that $|A\cap (A_n\setminus A_{n+1})| \le 1$ for each $n\in \omega$. An ultrafilter $p\in \omega^*$ is a 
\emph{$Q$-point}, or \emph{rare}, ultrafilter if, given any finite-to-one function $f\colon \omega \to \omega$, 
there exists an $A \in p$ such that the restriction of $f$ to $A$ is one-to-one. 

The proof of the following theorem is essentially contained in \cite{CN}.

\begin{theorem}
\label{Theorem2}
\textup{(i)} An ultrafilter $p\in \omega^*$ is minimal in the Rudin--Keisler order if and only if $p$ is 
selective. 

\textup{(ii)} An ultrafilter $p\in \omega^*$ is minimal in the Rudin--Blass order if and only if $p$ is 
a $Q$-point.

\textup{(iii)} An ultrafilter $p\in \omega^*$ is minimal in the Rudin--Frol\'\i k order if and only if $p$ 
is a discretely weak $P$-point in $\omega^*$. 
\end{theorem}

\begin{proof}
Assertion (i) is Theorem~9.6 of~\cite{CN}. Assertion (ii) (as well as (i)) easily follows from  
definitions and Remark~\ref{Remark_add}.  Assertion (iii) is 
Lemma~16.14 of~\cite{CN}. \end{proof} 

\begin{corollary}
\label{Corollary1}
If $p, q\in \omega^*$, $(x_n)_{n\in \omega}$ is a discrete sequence of distinct 
points in $\omega^*$,  and $q=\pl_n x_n$, then  
$p$ is not a  discretely weak $P$-point. 
\end{corollary}

\begin{proof}
By Proposition~\ref{Proposition1} $p\leRF q$. If $p$ is a discretely weak $P$-point, then $p$ and $q$ are 
equivalent by Theorem~\ref{Theorem2}\,(iii), so that $q$ is a discretely weak $P$-point as well. But a discretely 
weak $P$-point cannot be limit for a discrete set. 
\end{proof}

In the class of $P$-points  $\leRK$ coincides with $\leRB$. Moreover, the following assertion holds. 

\begin{proposition}
\label{proposition-added1}
If $p,q\in \omega^*$ and $p$ is a $P$-point, then $q\leRK p$ if and only if~$q\leRB p$.
\end{proposition}

\begin{proof}
Only the `only if' part needs to be proved. Suppose that $q\leRK p$ and let $f\colon \omega\to \omega$ be a 
function for which $\beta f(p)=q$. Take $A\in p$ such that the restriction of $f$ to $A$ is finite-to-one and 
both sets $\omega\setminus A$ and  $\omega\setminus f(A)$ are infinite. Let $\varphi$ be any bijection between 
$\omega\setminus A$ and $\omega\setminus f(A)$. The function $g\colon \omega\to \omega$ coinciding with $f$ on $A$ and with $\varphi$  
on $\omega\setminus A$ is finite-to-one. Obviously, $\beta g(p)=q$. Therefore, $q\leRB p$.
\end{proof}

The following amazing theorem of van~Mill shows that the relation $\leRF$ is very much smaller than 
$\leRK$. 

\begin{theorem}[{see \cite[Theorem~4.5.1]{vanMill}}]
\label{vanMillstheorem}
There is a finite-to-one function $\pi\colon \omega\to\omega$ such that,  given any $p\in \omega^*$, 
there is a weak $P$-point $q\in \omega^*$ for which $\beta\pi (q)=p$ \textup(and hence $p\leRB q$ and $p\leRK 
q$\textup). 
\end{theorem}

It is well known that there exist $2^{2^\omega}$ $\leRK$-incomparable (and hence $\leRB$- and 
$\leRF$-incomparable) ultrafilters in $\omega^*$ \cite{Rudin-Shelah}. However, the problem of the existence of 
\emph{$\leRK$-incompatible} (i.e., having no common $\leRK$-predecessor)  ultrafilters is much more complicated. 
On the one hand, CH implies the existence of $2^{2^\omega}$ nonequivalent $\leRK$-minimal ultrafilters in 
$\omega^*$ \cite[Sec.~8, Corollary~8]{Blass-thesis}. Clearly, such ultrafilters cannot be compatible in any of the 
orders $\leRK$, $\leRB$ and $\leRF$. On the other hand, it is consistent with ZFC that all ultrafilters in 
$\omega^*$ are \emph{nearly coherent}, i.e., even $\leRB$-compatible \cite{NCF}. Finally, there exist (in ZFC) at 
least $2^\omega$ nonequivalent (and even $\leRK$-incomparable) weak $P$-points in $\omega^*$ 
\cite{Kunen1978}, and such points are $\leRF$-minimal and hence $\leRF$-incompatible. 

\section{Discrete Ultrafilters}

In \cite{Baumgartner} Baumgartner introduced the notion of an $I$-ultrafilter and related classes of ultrafilters. 

\begin{definition}[\cite{Baumgartner}]
Let $I$ be a family of subsets of a set $X$ such that $I$ contains all
singletons and is closed under taking subsets. An ultrafilter $p$ on $\omega$ 
is said to be an \emph{$I$-ultrafilter} if, for any $f\colon \omega \to X$, there is an $A \in p$ such
that $f(A) \in  I$.

In the case where $X=\mathbb R$ and $I$ is the family of all discrete (scattered, measure zero, nowhere dense) 
subsets of $\mathbb R$, an $I$-ultrafilter is said to be \emph{discrete} (respectively, \emph{scattered}, 
\emph{measure zero}, \emph{nowhere dense}). 
\end{definition}

\begin{remark}
\label{P-discrete}
Baumgartner also proved that \emph{if $I=\{Y\subset 2^\omega: Y$ is finite or has order type of $\omega$ or 
$\omega+1\}$, then the nonprincipal $I$-ultrafilters are exactly the $P$-points of $\omega^*$.} (Here $2^\omega$ 
is the Cantor set with the lexicographic order.) This immediately implies that \emph{any $P$-point is discrete}. 
\end{remark}

Thus, we have 
$$ 
\text{$P$-point$\implies$discrete$\implies$scattered$\implies$measure zero$\implies$nowhere dense}.
$$
Under Martin's axiom none of these implications reverses~\cite{Baumgartner}. It makes no sense to speak 
about their reversibility without additional set-theoretic assumptions, because the nonexistence of nowhere 
dense ultrafilters is consistent with ZFC~\cite{Shelah}.  

Considering families $I$ of discrete subsets of other spaces $X$ and imposing 
assumptions on $f\colon \omega\to X$, we obtain potentially different classes of discrete-like ultrafilters. 

\begin{definition}
Let $X$ be a space. We say that an ultrafilter $p$ on $\omega$ is 
\emph{$X$-discrete} (\emph{finitely-to-one $X$-discrete}, \emph{injectively $X$-discrete}) if,  for any  
(respectively, for any finite-to-one, for any one-to-one) function  $f\colon \omega \to X$, there is an $A \in p$ 
such that $f(A)$ is discrete in~$X$. For $X=\mathbb R$, we write simply ``discrete'' instead of ``$\mathbb 
R$-discrete.'' 
\end{definition}

\begin{remark}
\label{Remark2}
Note that any $\omega^*$-discrete (in any sense) 
ultrafilter $p$ is $\beta\omega$-discrete (in the same sense). Indeed, take any $f\colon \omega\to 
\beta\omega $. If $A=f^{-1}(\omega)\in p$, then $f(A)\subset\omega\subset \beta\omega$ is discrete; otherwise $B=\omega\setminus A\in 
p$, and we can fix any distinct $q_n\in \omega^*\setminus f(B)$, $n\in \omega$, and consider the map $g\colon 
\omega\to \omega^*$ defined by $g(n)=f(n)$ for $n\in B$ and $g(n)=q_n$ for $n\in A$. Let $C\in p$ 
be such that $g(C)$ is discrete. Then $C\cap B\in p$ and $f(C\cap B)=g(C\cap B)$ is discrete. 
\end{remark}
 
Injectively discrete and injectively $\omega^*$-discrete ultrafilters were considered in \cite{BB}  (where they 
were called simply ``discrete'' and ``$\omega^*$-discrete''). The following proposition is 
similar to Proposition~12 of~\cite{BB}. 

\begin{proposition}
Every discrete \textup(finitely-to-one discrete, injectively discrete\textup) ultrafilter is $X$-discrete 
\textup(respectively, finitely-to-one $X$-discrete, injectively $X$-discrete\textup) for any space~$X$.
\end{proposition}

\begin{proof}
Let $p$ be  a discrete ultrafilter, and let $f \colon \omega\to X$ be any map. We set $x_n=f(n)$ for 
$n\in \omega$. For each pair $(x_k, x_m)$ of different points in $f(\omega)$, take a continuous function 
  $f_{(k,m)}\colon f(\omega)\to \{0,1\}$ such that $f_{(k,m)}(x_k)=0$ and $f_{(k,m)}(x_m)=1$ (it exists because 
$f(\omega)$ is countable). The diagonal  $g=\Delta \{f_{(k,m)}: x_k\ne x_m\}$ is a one-to-one continuous map 
of $f(\omega)$ to the Cantor set $2^\omega\subset \mathbb R$. Since $p$ is discrete, it follows that there is an 
$A\in p$ for which $g(f(A))$ is discrete. Thus, $f(A)$ admits a one-to-one continuous map onto a discrete space;
therefore, $f(A)$ is discrete. For finitely-to-one and injectively discrete ultrafilters, the proof is the same. 
\end{proof}

Thus, any discrete ultrafilter is $X$-discrete for any space $X$, but it is unclear whether, say, an 
$\omega^*$-discrete ultrafilter is discrete. The questions of whether injective discreteness implies 
finite-to-one discreteness and whether finite-to-one discreteness implies discreteness are not clear either, 
although the answer is unlikely to be positive. However, the nonexistence of injectively and finitely-to-one  
discrete ultrafilters, as well as that of discrete ones, in consistent with ZFC, because it follows from the 
nonexistence of nowhere dense ultrafilters. The following argument was kindly communicated to the authors by Taras 
Banakh. 

\begin{proposition}
The nonexistence of nowhere dense ultrafilters implies the nonexistence of 
injectively discrete ultrafilters.
\end{proposition}

\begin{proof}
Let us say that an ultrafilter $p$ on $\omega$ is injectively nowhere dense if,  for any  
one-to-one function  $f\colon \omega \to \mathbb R$, there is an $A \in p$ 
such that $f(A)$ is nowhere dense. We will prove that the nonexistence of nowhere ultrafilters implies that of 
injectively nowhere dense ultrafilters. 

Suppose that there exist no nowhere dense ultrafilters but there exists and injectively nowhere dense ultrafilter 
$p$. Since $p$ is not nowhere dense, it follows that we can find a function $f\colon \omega \to 
(0,1)\cong \mathbb R$ such that, for any $A\in p$, there exists an open set 
$U\subset (0, 1)$ in which $U\cap f(A)$ is dense. Take an injective function $g\colon \omega \to \mathbb R$ such 
that $|g(n)-f(n)|<2^{-n}$ for all $n$. Since $p$ is injectively nowhere dense, there exists an $A\in p$ whose 
image $g(A)$ is nowhere dense in $\mathbb R$. On the other hand, the choice of $f$ ensures that, for some open set 
$U$ in $R$, the intersection $U\cap f(A)$ is dense in $U$. Thus, the set $U\cap f(A)$ has no isolated points. 
The condition $|g(n)-f(n)|\to 0$ implies that  $U\cap g(A)$ is dense in $U$, and hence $g(A)$ cannot be 
nowhere dense in~$\mathbb R$. 
\end{proof}

\begin{proposition}
\label{Proposition4}
\textup{(i)}\enspace 
If $p\in \beta\omega$ is $X$-discrete for some space $X$ and $q\leRK p$, then $q$ is $X$-discrete.

\textup{(ii)}\enspace 
If $p\in \beta\omega$ is finitely-to-one $X$-discrete for some space $X$ and $q\leRB p$, then $q$ is 
finitely-to-one $X$-discrete.

\textup{(iii)}\enspace 
If $p\in \beta\omega$ is $\omega^*$-discrete and $q$ is the nontrivial $p$-limit of 
some sequence $(x_n)_{n\in\omega}$ in $\beta\omega$, then there exists an $r\in \beta\omega$ such that 
$p\leRK r \leRF q$. Moreover, $q=r$-$\lim_n x_{k_n}$ for some discrete subsequence $(x_{k_n})_{n\in \omega}$ of 
$(x_n)_{n\in \omega}$ consisting of distinct points. 

\textup{(iv)}\enspace 
If $p\in \beta\omega$ is injectively $\omega^*$-discrete, then $p\leRF q$ if and only if $q$ is the $p$-limit 
of some sequence $(x_n)_{n\in\omega}$ of distinct points of~$\beta\omega$. 
\end{proposition} 

\begin{proof}
The first two assertions are obvious, as well as the `only if' part of the fourth one. 

Let us prove (iii). Suppose that $(x_n)_{n\in 
\omega}$ is any sequence in $\beta\omega$ and $q$ is the nontrivial $p$-limit of $(x_n)_{n\in \omega}$. If the 
ultrafilter $p$ is principal, then there is nothing to prove, so we will assume that $p\in \omega^*$. Recall that, 
by Remark~\ref{Remark2}, the $\omega^*$-discreteness of the ultrafilter $p$ implies its 
$\beta\omega$-discreteness. Let $A\in p$ be such that the set $\{x_n: n\in A\}$ is discrete. This set is infinite, 
because $q$ is the nontrivial $p$-limit of $(x_n)_{n\in \omega}$. Let $(x_{k_n})_{n\in \omega}$ be a subsequence 
of $(x_n)_{n\in \omega}$ consisting of distinct points and such that $\{x_n: n\in A\}=\{x_{k_n}: n\in \omega\}$. 
Note that $(x_{k_n})_{n\in \omega}$ is discrete. 

Take any function $\pi\colon \omega\to \omega$ such that, for every $i\in A$, $\pi(i)=n$ if and only if 
$x_i=x_{k_n}$. The ultrafilter $r=\beta\pi(p)$ is a $\leRK$-successor of $p$, and for each neighborhood $U$ of 
$q$ in $\beta\omega$, we have $\{n: x_{k_n}\in U\}=\pi(\{n: x_n\in U\})\in r$. Therefore, $q= r$-$\lim_n x_{k_n}$. 
By Remark~\ref{Remark1} we have $r\leRF q $.

The proof of the `if' part of  assertion (iv) is similar, the only difference being that we must consider a 
one-to-one sequence $(x_n)_{n\in\omega}$; then the restriction of $\pi$ to $A$ is one-to-one, so that $p$ is 
equivalent to $r$ (see the argument in the proof of Theorem~\ref{Theorem2}). 
\end{proof}

\section{Products of Ultrafilters}

Recall that the \emph{tensor}, or \emph{Fubini}, \emph{product} $p \otimes q$ of ultrafilters 
$p$ and $q$ on $\omega$ is the ultrafilter on $\omega\times \omega$ defined by 
$$
p \otimes q=\{A \subseteq \omega \times \omega:
\{n:\{m:(n, m) \in A\} \in q\} \in p\} 
$$ 
(see, e.g., \cite{HS}). It is generated by a base consisting of sets of the form 
$$
\bigcup_{n\in P}\{n\}\times Q_n, \qquad \text{where $P\in p$ and $Q_n\in q$ for each $n\in P$}.
$$ 

A generalization of the Fubini product of two ultrafilters is the \emph{Fubini sum} $\sum_p(q_n)$ of a sequence 
of ultrafilters $q_n\in \beta\omega$ over $p\in \beta \omega$, which is generated by sets of the form 
$$
\bigcup_{n\in P}\{n\}\times Q_n, \qquad \text{where $P\in p$ and $Q_n\in q_n$ for each $n\in P$}.
$$

Considering products of ultrafilters, we assume that $\omega\times \omega$ is endowed with the discrete topology, 
so that the space $\beta(\omega\times \omega)$ consists of ultrafilters on $\omega\times \omega$  and its 
topology is generated by the base of sets of the form $\overline A=\{p\in \beta (\omega\times \omega): A\in p\}$. 
Thus, the spaces $\beta \omega$ and $\beta (\omega\times \omega)$ are homeomorphic and have the same description, 
and all notions and constructions and related to $\beta\omega$ carry over to $\beta(\omega\times \omega)$ without 
any changes. In what follows, we identify $\omega$ with $\omega\times \omega$ and $\beta \omega$ with 
$\beta(\omega\times \omega)$ when appropriate. 

In \cite{vanMill16} van~Mill defined various 
topological types of ultrafilters in $\omega^*$, one of which was 
$$ 
A_1 = \{x\in\omega^* : 
\exists\,\text{countable discrete $D\subset\omega^*\setminus\{x\}$ with $x\in\overline D$}\}. 
$$ 
Thus, ultrafilters 
of van~Mill's type $A_1$ are precisely those $p\in \omega^*$ which are not discretely weak $P$-points.

\begin{proposition}
\label{Proposition5}
The tensor product $p\otimes q$ of any ultrafilters $p,q\in \omega^*$ belongs to $A_1$. 
\end{proposition}

\begin{proof}
Recall that, for $n\in \omega$, $p(n)$ denotes the principal ultrafilter on $\omega$ generated by $\{n\}$. 
Let $r= \pl_n (p (n) \otimes q)$. Any neighborhood of $r$ in $\beta(\omega\times \omega)$ contains a 
neighborhood of the form $\overline A$ for $A\in r$. Since $r= \pl_n (p (n) \otimes q)$, for each $A\in r$, we 
have $B = \{n : A\in p(n) \otimes q\}\in p$. Hence, for every $n\in \omega$, there exists a 
 $B_n\in q$ such that $\{n\}\times B_n\subset A$. Thus, each $A\in r$ 
contains $\bigcup\limits_{n\in B}\{n\}\times B_n$ for some $B_n\in q$. It follows from the definition of the base 
of a product of ultrafilters that $r = p \otimes q$. 
Note that the set $D = \{p(n) \otimes q : n\in\omega\}$ is countable and discrete (disjoint neighborhoods of its 
elements are $\overline{\{n\}\times B_n}$) and $p\otimes 
q\in\overline{D}$, because any neighborhood of $p\otimes q$ contains a neighborhood of the form 
$\overline{A}$ for $A\in p\otimes q$, any $A$ contains $B=\{n\}\times B_n$ for some $n\in \omega$ and $B_n\in q$, 
and $\overline B$ is a neighborhood of $p(n)\otimes q$ for any such $\overline B$. 
\end{proof} 

\begin{corollary}
A tensor product of two nonprincipal ultrafilters is never a discretely weak $P$-point. Thus, any such product is 
$\leRF$-minimal.
\end{corollary}

\begin{proposition}
For any compact space $X$ and any $X$-discrete (finitely-to-one 
$X$-discrete, injectively $X$-discrete) ultrafilters $p,q_n\in \omega^*$, $n\in \omega$, the Fubini sum $\sum_p 
(q_n)$ is an $X$-discrete (finitely-to-one $X$-discrete, injectively $X$-discrete) ultrafilter 
on $\omega\times \omega$. 
\end{proposition}

\begin{proof} 
Take any (any finite-to-one, any one-to-one) sequence $(x_{(n,m)})\subset X$. We must show that there 
exists an $A\in \sum_p (q_n)$ for which the set $\{x_{(n,m)}: (n,m)\in A\}$ is discrete. For each $k\in\omega$, 
consider the sequence $(x_{(k,m)})_{m\in\omega}$. Since $q_k$ is discrete, it follows that, for each $k\in 
\omega$, there exists a $B_k\in q_k$ for which the set $\{x_{(k,m)} : m\in B_k\}$ is discrete. Recall that, in a 
compact space, any sequence has a limit along any ultrafilter. Let $x_k=q_k$-$\lim_m x_{(k,m)}$ for $k\in 
\omega$. The $X$-discreteness of $p$ implies the existence of a $C\in p$ for which the set $\{x_k : k\in C\}$ is 
discrete. Since $\{x_k : k\in C\}$ is countable, it is strongly discrete, that is, there exists a disjoint system 
of neighborhoods $U_k$ of the points $x_k$ in $X$ (such a system is easy to construct by induction). For each 
neighborhood $U_k$, we have $\tilde B_k=\{m: x_{(k,m)} \in U_k\}\in q_k$, because $x_k=q_k$-$\lim_m x_{(k,m)}$. 
Thus, $A = \bigcup_{k\in C}\bigl(\{k\}\times (B_k\cap \tilde{B}_k)\bigr)\in \sum_p (q_k)$. Clearly, the set 
$\{x_{(n,m)}: (n,m)\in A\}$ is discrete. 
\end{proof} 

\begin{corollary}
For any compact space $X$ and any $X$-discrete (finitely-to-one 
$X$-discrete, injectively $X$-discrete) ultrafilters $p,q\in\omega^*$, the tensor product $p 
\otimes q$ is an $X$-discrete (finitely-to-one $X$-discrete, injectively $X$-discrete) ultrafilter 
on $\omega\times \omega$. 
\end{corollary}

Below we consider the set $\mathbb N$ of positive integers instead of $\omega$ solely in order that multiplication 
be a finite-to-one map. On the space $\beta \mathbb N$ the semigroup operations $\cdot$ and $+$ are defined (see 
\cite{HS}). Given two ultrafilters $p$ and $q$ on $\mathbb N$, their semigroup product $p\cdot q$ and sum $p+q$ 
are generated, respectively, by the sets 
$$ 
\bigcup_{n\in P}(n \cdot Q_n) \;\text{ and }\; \bigcup_{n\in P}(n + 
Q_n), \quad \text{where $P\in p$ and $Q_n\in q$ for each $n\in P$} 
$$ 
(here $\cdot$ and $+$ denote the usual 
multiplication and addition in $\mathbb N$). The maps 
$$ 
\cdot\colon \mathbb N\times \mathbb N\to \mathbb N,\; (m,n)\mapsto m\cdot n, \quad \text{and} \quad 
+\colon \mathbb N\times \mathbb N\to \mathbb N,\; (m,n)\mapsto m+n, 
$$ 
are finite-to-one. Clearly, $\cdot(p\otimes q)=p\cdot q$ and $+(p\otimes q)=p\cdot q$ for any $p,q\in \beta\mathbb 
N$. We obtain the following corollary.

\begin{corollary}
For any compact space $X$ and any $X$-discrete (finitely-to-one 
$X$-discrete) ultrafilters $p,q\in\mathbb N^*$, the ultrafilters $p\cdot q$ and  
$p+q$ are $X$-discrete (finitely-to-one $X$-discrete). 
\end{corollary}

\begin{corollary}
The sets of discrete, finitely-to-one discrete, $\omega^*$-discrete, and finitely-to-one $\omega^*$-discrete 
ultrafilters on $\mathbb N$ are subsemigroups in the semigroups $(\beta\mathbb N,\cdot)$ and $(\beta\mathbb 
N,+)$.
\end{corollary}

\begin{corollary}
If there exist discrete ultrafilters, then there exist discrete ultrafilters which are not discretely weak 
$P$-points (and hence are not $\leRF$-minimal). 
\end{corollary}

\section{Classes of Spaces Between $F$-Spaces and $\beta\omega$-Spaces} 

In what follows, we consider extremally disconnected, $F$-, and $\beta\omega$-spaces. A space $X$ 
is \emph{extremally disconnected} if any disjoint open sets in $X$ have disjoint closures (or, equivalently, 
if the closure of any open set is open), and $X$ is called an \emph{$F$-space} if any disjoint cozero sets in $X$ 
are completely (= functionally) separated. Clearly, any extremally disconnected space is an $F$-space. The basic 
topological properties of extremally disconnected spaces and $F$-spaces can be found in the fundamental book 
\cite{GJ} by Gillman and Jerison. The class of $\beta\omega$-spaces was introduced by van~Douwen \cite{vanDouwen} 
as a generalization of the class of $F$-spaces. A space $X$ is called a \emph{$\beta\omega$-space} if, 
whenever $D$ is a countable discrete subset of $X$ with compact closure $\overline D$ in $X$, its closure 
$\overline D$ is the Stone--\v Cech compactification~$\beta D$ (or, equivalently, $\overline D$ is homeomorphic 
to $\beta\omega$). 

It is known that any countable separated sets in an extremally disconnected space are functionally separated 
\cite[1.6]{Frolik2}. It follows that any countable subspace of an extremally disconnected space is extremally 
disconnected. Moreover, any countable subspace of an $F$-space is extremally disconnected as well. 
Indeed, according to \cite[9H.1]{GJ}, any countable set in an $F$-space is $C^*$-embedded, and it is easy to see 
that the property of being an $F$-space is inherited by $C^*$-embedded subspaces. Thus, any countable subspace of 
an $F$-space is an $F$-space. It remains to note that all countable spaces are perfectly normal, so that any open 
set in such a space is a cozero set. 

Note also that all countable subspaces of a space $X$ are extremally disconnected if and only if any 
countable separated subsets of $X$ have disjoint closures. Indeed, suppose that all countable subspaces  of 
$X$ are extremally disconnected and let $A$ and $B$ be countable separated subsets of $X$. Then $A\cup B$ is 
extremally disconnected, and $A$ and $B$ are separated in $A\cup B$. According to \cite[Proposition~1.9]{Frolik2}, 
we have $\overline A\cap \overline B=\varnothing$. Conversely, suppose that countable separated subsets of $X$ 
have disjoint closures and let $Y$ be a countable subspace of $X$. Obviously, any disjoint open subsets of $Y$ are 
separated in $Y$ and hence in $X$. Therefore, they have disjoint closures in $X$ and hence in $Y$. 

These observations suggests a number of natural generalizations of the class of $F$-spaces. 

\begin{definition}
We say that a topological space $X$ is 
\begin{itemize}
\item
an \emph{$\cR_1$-space} if any countable subset of $X$ is extremally disconnected, 
i.e., any two separated countable subsets of $X$ have disjoint closures; 
\item
an \emph{$\cR_2$-space} if any two separated countable subsets of $X$ one of which is discrete have disjoint 
closures; 
\item
an \emph{$\cR_3$-space} if any two separated countable discrete subsets of $X$ have disjoint 
closures.
\end{itemize} 
\end{definition}

Importantly, the classes of $\cR_i$-spaces are hereditary, unlike those of extremally disconnected 
and $F$-spaces.

It is clear that 
$$
\text{$F$-spaces}\subset\text{$\cR_1$-spaces}\subset \text{$\cR_2$-spaces}\subset \text{$\cR_3$-spaces}.
$$ 
However, the reverse 
inclusions do not hold. Examples distinguishing between these classes are the quotient spaces $(\beta\omega\oplus 
\beta\omega)/\{p, q\}$, where $p$ belongs to the first copy of $\beta\omega$ and $q$, to the second one. An 
example of an $\mathscr R_1$-space which is not an $F$-space is obtained when both $p$ and $q$ are weak $P$-points 
not being $P$-points. An example of an $\mathscr R_2$-space which is not an $\mathscr R_1$-space is obtained when 
both $p$ and $q$ are discretely weak $P$-points not being weak $P$-points. Finally, an example of an $\mathscr 
R_3$-space which is not an $\mathscr R_2$-space is obtained when $p$ is a discretely weak $P$-point not being a 
weak $P$-point and $q$ is not a discretely weak $P$-point. For details, see~\cite{GS}.

\begin{proposition}
\textup{(i)}\enspace 
A space $X$ is an $\cR_3$-space if and only if any countable discrete set $D\subset X$ is 
$C^*$-embedded in~$\overline D$. 

\textup{(ii)}\enspace 
A space $X$ is an $\cR_3$-space if and only if the closure of any countable discrete set 
$D\subset X$ in $\beta X$ is the Stone--\v Cech compactification $\beta D$ of~$D$. 

\textup{(iii)}\enspace 
A space $X$ is a $\beta\omega$-space if and only if any countable discrete set $D\subset X$ 
with compact closure is $C^*$-embedded in~$\overline D$. 
\end{proposition}

\begin{proof}
(i)\enspace 
First, note that countable discrete sets $A, B\subset X$ are separated if and only if $D=A\cup B$ is 
discrete. Therefore, $X$ is a $\cR_3$-space if and only if any disjoint subsets of any discrete set $D\subset X$ 
have disjoint closures. By Taimanov's theorem (see \cite[Theorem~3.2.1]{Engelking}) this means precisely that $D$ 
is $C^*$-embedded in the closure of $D$ in $\beta X$. 

Assertions~(ii) and (iii) immediately follow from (i) and the definitions of Stone--\v Cech compactification and 
of a $\beta\omega$-space. 
\end{proof}

\begin{corollary}
Any $\mathscr R_3$-space is a $\beta\omega$-space. Any compact $\beta\omega$-space is an $\mathscr R_3$-space. 
\end{corollary}

The class of $\beta\omega$-spaces is strictly larger than that of $\cR_3$-spaces: any space containing no infinite 
compact subspaces is a $\beta\omega$-space but not necessarily a $\cR_3$-space. From some point of view, the 
property of being an $\cR_3$-space is more natural than that of being a $\beta\omega$-space. 

A property which is in a sense opposite to $\cR_3$ was introduced by Kunen in \cite{K-OPIT}. He called a space 
$X$ \emph{sequentially small} if any infinite set in $X$ has an infinite subset whose closure does not contain a 
copy of $\beta\omega$. Thus, a compact space $X$ is sequentially small if none of its countable discrete subsets 
is $C^*$-embedded. 

\section{Homogeneity in Product Spaces}

In \cite{GS} we extended Kunen's lemma cited at the beginning of this paper as follows. 

\begin{proposition}[\cite{GS}]
\label{proposition-Kunen}
Let $X$ be any compact $\cR_2$-space. Suppose that $x\in X$,  $(d_m)_{m\in \omega}$ is a discrete sequence of 
distinct points in $X$, $(e_n)_{n\in \omega}$ is any sequence of points in $X$, and $x = \pl_m d_m = \ql_n e_n$, 
where $p$ is a weak $P$-point in $\omega^*$ and $q$ is any point in $\omega^*$. If $\{n : e_n = x\} \notin q$, 
then $p\leRK q$. 
\end{proposition}

Imposing additional constraints on ultrafilters, we can further extend the class of spaces to which Kunen's lemma 
applies. 

\begin{proposition}
\label{proposition-Kunen2}
Let $X$ be
any compact $\beta\omega$-space. Suppose that $x\in X$,  $(d_m)_{m\in \omega}$ is a discrete sequence of distinct 
points in $X$, $(e_n)_{n\in \omega}$ is any sequence of points in $X$, and $x = \pl_m d_m = \ql_n e_n$, where 
$p,q\in \omega^*$, $p$ is a discretely weak $P$-point in $\omega^*$, and $q$ is discrete. If $\{n : e_n = x\} 
\notin q$, then $p\leRK q$. 
\end{proposition}

\begin{proof}
Since $q$ is discrete and $\{n : e_n = x\} 
\notin q$, it follows that there exists a $Q\in q$ for which the set $E=\{e_n:n\in 
Q\}$ is discrete (and $x\in \overline E\setminus E$). By assumption $D=\{d_m:m\in \omega\}$ is discrete as 
well, and $x\in \overline D\setminus D$.  Since $X$ is a compact $\beta\omega$-space, the point $x$ has a 
neighborhood $U$ such that either $U\cap (D\setminus \overline E)=\varnothing$ or $U\cap (E\setminus \overline 
D)=\varnothing$ (otherwise $x$ would belong to the intersection of the closures of the separated countable 
discrete sets $E\setminus \overline D$ and $D\setminus \overline E$). By the definition of $q$-limit, the set 
$\{n\in \omega: e_n\in U\}$ belongs to $q$. We assume without loss of generality that $Q$ is contained in this 
set. 

Since $x\in \overline D\cap \overline E$ and $x\notin D\setminus \overline E\cup E\setminus \overline D$, 
we have either $x\in \overline {\overline D\cap E}$ or $x\in \overline{D\cap \overline E}$. 

Suppose that $x\in \overline{\overline D\cap E}\subset \overline D$. Clearly, we then have $U\cap (E\setminus 
\overline D)=\varnothing$. Recall that $\overline D=\beta D$ (because $X$ is a compact $\beta\omega$-space) and 
consider the map $f\colon d_m\mapsto m$. We have $\beta f(x)=\pl_m \beta f(d_m)=\pl_m m=p$ (see 
Remark~\ref{Remark1}). On the other hand, setting $e'_n=e_n$ for $n\in Q$ and $e'_n=x$ 
for $n\in \omega\setminus Q$, we obtain a sequence $(e'_n)_{n\in \omega}$ for which $x=\ql_n e'_n$, because 
the sequence $(e'_n)$ coincides with $(e_n)$ on an element of $q$. Therefore, 
$p=\beta f(x)=\ql_n \beta f(e'_n)$ by Remark~\ref{Remark1}\,(ii), 
and  Proposition~\ref{Proposition4}\,(iii) implies $p\leRK q$. 

Now suppose that $x\in \overline{D\cap \overline E}\subset \overline E$; in this case, $U\cap (D\setminus 
\overline E)=\varnothing$, so that $\{m\in \omega: d_m\in \overline E\}\in p$. Let us somehow number the points 
of $E$ as $\{e'_n:n\in\omega\}$, so that $(e'_n)_{n\in \omega}$ is a discrete sequence of distinct points with 
range $E$, and define $\varphi\colon \omega\to \omega$ by setting $\varphi(n)$ equal to, say, $0$ for $n\in 
\omega\setminus Q$ and to the number $k$ such that $e_n=e'_k$ for $n\in Q$. It is easy to check that 
$x=\beta\varphi(q)$-$\lim_n e'_n$ (see Remark~\ref{Remark1}\,(iii)). Consider the one-to-one map $g\colon E\to 
\omega$ defined by $g(e'_n) = n$. We have $\overline E=\beta E$ and $\beta g(x)=\beta\varphi(q)$-$\lim_n\beta 
g(e'_n)=\beta \varphi(q)$.

Suppose that $\{m: d_m\notin E\}\in p$. Then 
there are $P,P'\in p$, $P\subset P'$, for which   
$\{d_m:m\in P'\}\subset E^*=\beta E\setminus E$ and $P'\setminus P$ is infinite. Choose a bijection 
$\Psi\colon P'\setminus P\to \omega \setminus P$. Setting $d'_m=d_m$ for $m\in P$ and $d'_m=d_{\Psi^{-1}(m)}$ for 
$m\in \omega\setminus P$, we obtain a new discrete sequence $(d'_m)_{m\in \omega}$ of distinct points with range 
$D'\subset D\cap E^*$ which coincides with $(d_m)_{m\in \omega}$ on $P\in p$. Clearly, we still have $x= \pl_m 
d'_m$ and $\beta g(x)=\pl_m \beta g(d'_m)$; moreover, $\beta g(d'_m)\in \omega^*$ for all $m$. But this is 
impossible by Corollary~\ref{Corollary1}.

Thus, there exists a $P'\in p$ for which $\{d_m:m\in P'\}\subset E$. For 
the sequence $(d'_m)_{m\in \omega}$ constructed in precisely the same way as above (by taking $P\in p$, $P\subset  
P'$, such that  $P'\setminus P$ is infinite and redefining $(d_m)$ on $\omega \setminus P$), we have 
$D'=\{d'_m:m\in \omega\}\subset E$ and hence $(\beta g)\restriction D'=g\restriction D'$. Note also that 
$d'_m=d_m$ for  $m\in P$.

Since the element $P$ of $p$ has infinite complement in $\omega$, there is a bijection $\psi\colon 
\omega\to \omega$ such that $\psi(g(d'_m))=m$ for $m\in P$. Then the sequence $(m)_{m\in 
\omega}$, which coincides with $(\psi(g(d'_m)))_{m\in \omega}$ and hence with $(\psi(g(d_m)))_{m\in \omega}$ 
when restricted to $P$, converges 
to $\beta\psi(\beta g(x))$ along $p$. Therefore, $\beta\psi(\beta g(x))=p$. Since $\psi$ is one-to-one, 
it follows that $p$ is equivalent to $\beta g(x)$, and since $\beta g(x)= \beta\varphi(q)$, it follows that 
$p\leRK q$. 
\end{proof}

\begin{remark}
\label{remark-added2}
For any ultrafilter $q\in \omega^*$, there exists a weak $P$-point $p\in \omega^*$ such that  
$p\not \leRK q$. 

Indeed, by Remark~\ref{remark-added1} \,$q$ has at most $2^\omega$ \,$\leRK$-predecessors, 
while the number of weak $P$-points in $\omega^*$ is $2^{\omega^\omega}$~\cite{Kunen1978}.
\end{remark}

\begin{corollary}
\label{corollary-added}
If there exists a discrete ultrafilter in $\omega^*$, then there exist no homogeneous compact 
$\beta\omega$-spaces. 
\end{corollary}

\begin{proof}
Let $q$ be a discrete ultrafilter in $\omega^*$, and let $p\in \omega^*$ be a weak $P$-point such that   
$p\not \leRK q$. Suppose that $X$ is a homogeneous compact $\beta\omega$-space and $(e_n)_{n\in \omega}$ is 
any sequence of distinct points in $X$. Let $x = \pl_n e_n$, and let $y= \ql_n e_n$. Since $X$ is 
homogeneous, there exists a homeomorphism $ h\colon X\to X$ taking $y$ to $x$, and since 
$q$ is discrete, there is an $A\in q$ for which $\{h(e_n):n\in A\}$ is discrete. We can assume without loss of 
generality that $(h(e_n))_{n\in \omega}$ is a discrete sequence of 
distinct points. By Proposition~\ref{proposition-Kunen2} we have $p\leRK q$, which contradicts the assumption. 
\end{proof}

Kunen used his lemma to prove a theorem on the nonhomogeneity of product spaces~\cite[Theorem~1]{K-OPIT}. 
Using Propositions~\ref{proposition-Kunen} and~\ref{proposition-Kunen2} and Remark~\ref{remark-added2} 
instead of the lemma in Kunen's 
argument, we obtain the following results. 

\begin{theorem}
Let  $X=\prod_{\alpha<\kappa}X_\alpha$, where $\kappa$ is any cardinal and each $X_\alpha$ satisfies at 
least one of the following conditions: \textup{(i)}~is an infinite compact $\cR_2$-space; \textup{(ii)}~contains a 
weak $P$-point; \textup{(iii)}~has a nonempty sequentially small open subset. Suppose also that at least one 
$X_\alpha$ is an infinite compact $\cR_2$-space. Then $X$ is not homogeneous. 
\end{theorem}

\begin{corollary}
No product of compact $\cR_2$-spaces is homogeneous.
\end{corollary}

\begin{theorem}
Suppose that there exists a discrete ultrafilter in $\omega^*$. 
Let  $X=\prod_{\alpha<\kappa}X_\alpha$, where $\kappa$ is any cardinal and each $X_\alpha$ satisfies at 
least one of the following conditions: \textup{(i)}~is an infinite compact $\beta\omega$-space; 
\textup{(ii)}~contains a 
weak $P$-point; \textup{(iii)}~has a nonempty sequentially small open subset. Suppose also that at least one 
$X_\alpha$ is an infinite compact $\beta\omega$-space. Then $X$ is not homogeneous. 
\end{theorem}

\begin{corollary}
If there exists a discrete ultrafilter in $\omega^*$, then no product of compact $\beta\omega$-spaces 
is homogeneous. 
\end{corollary}

The following corollary uses the assumption $\mathfrak d = \mathfrak c$. Recall that the notation 
$\mathfrak d$ is used for the \emph{dominating number}, that is, the smallest cardinality of 
a family $\mathscr D$ of functions $\omega\to \omega$ with the property that, for every 
function $f\colon \omega \to \omega$, there is a $g\in \mathscr D$ such that $g(n)\ge f(n)$ for all but finitely 
many $n\in \omega$,  and $\mathfrak c$ is the standard notation for $2^\omega$. Obviously, $CH$ implies 
$\omega_1=\mathfrak d =\mathfrak c$, although $\omega_1< \mathfrak d = \mathfrak c$ is consistent with ZFC 
as well.

\begin{corollary}
Under the assumption $\mathfrak d = \mathfrak c$, no product of compact $\beta\omega$-spaces 
is homogeneous. 
\end{corollary}

\begin{proof}
Ketonen proved that $\mathfrak d = \mathfrak c$ implies the existence of 
$P$-points in $\omega^*$ \cite{Ketonen}. By Remark~\ref{P-discrete} any $P$-point 
is a discrete ultrafilter.  
\end{proof}

In conclusion, we mention recent results of Reznichenko concerning homogeneous compact subspaces of product 
spaces. He proved that, under CH, (i)~any compact set in a homogeneous subspace of a countable product of 
$\beta\omega$-spaces is metrizable, (ii)~any compact set in a homogeneous subspace of a finite product of 
$\beta\omega$-spaces is finite~\cite[Theorems~3 and~4]{Reznichenko}. An analysis of his proof shows that 
CH can be replaced by the assumption that there exist uncountably many 
$\leRB$-incompatible (that is, not nearly coherent) $P$-points. To be more precise, 
the following theorems hold. 

\begin{theorem}
Suppose that there exist uncountably many $\leRB$-incompatible $P$-points and $X=\prod_{n\in \omega}X_n$, where 
each $X_n$ is a compact $\beta\omega$-space. Let $Y\subset X$ be a homogeneous space. Then each compact subspace 
of $Y$ is metrizable. 
\end{theorem}

\begin{theorem}
Let $n$ be a positive integer. Suppose that there exist $n+1$ \,$\leRB$-incompatible $P$-points and 
$X=\prod_{i=1}^n X_i$, where 
each $X_i$ is a compact $\beta\omega$-space. Let $Y\subset X$ be a homogeneous space. 
Then each compact subspace 
of $Y$ is finite. 
\end{theorem}

This gives rise to the question of 
investigating conditions for the existence of uncountably many not nearly coherent $P$-points. 
A plausible conjecture is that such a condition is 
$\mathfrak d = \mathfrak u = \mathfrak c$ (here $\mathfrak u$ is the minimum cardinality of 
a free ultrafilter base on $\omega$), because under this condition there exist, first, $2^{2^\omega}$ 
Rudin--Keisler incomparable $P$-points~\cite{DT} and, secondly, $2^{2^\omega}$ 
near-coherence classes of ultrafilters~\cite{BB}.

\section*{Acknowledgments}

The authors are very grateful to Evgenii Reznichenko and Taras Banakh for helpful discussions.

\end{document}